\documentclass[11pt,reqno]{amsart}
	
	\usepackage{amssymb,amsmath,amsthm,amscd,latexsym,amsfonts}
	\usepackage{mathtools}
	\usepackage[T1]{fontenc}
	\usepackage{graphicx,epstopdf}
	\usepackage{graphicx}
	\usepackage{xcolor}
	\usepackage{cite}
	\usepackage{float}
	\usepackage[a4paper,top=3.1cm, bottom=3.1cm, left=3.3cm, right=3.3cm]{geometry}
	
	\newtheorem{thm}{Theorem}
	
	\newtheorem{lemma}{Lemma}
	
	\newtheorem{rk}{Remark}

	\numberwithin{equation}{section} 

	\newcommand{\bea}{\begin{eqnarray}}
		\newcommand{\eea}{\end{eqnarray}}

\begin{document}
		\title [Evolutionary behavior in a two-locus system]
		{Evolutionary behavior in a two-locus system}
		
		\author{A. M. Diyorov, U. A. Rozikov}
		
		\address{A. \ M. \ Diyorov \\ The Samarkand branch of TUIT, Samarkand, Uzbekistan.}
		\email {dabduqahhor@mail.ru}
		
		\address{ U.A. Rozikov$^{a,b,c}$\begin{itemize}
				\item[$^a$] V.I.Romanovskiy Institute of Mathematics,  9, Universitet str., 100174, Tashkent, Uzbekistan;
				\item[$^b$] AKFA University,
				264, Milliy Bog street,	Yangiobod QFY, Barkamol MFY,
				Kibray district, 111221, Tashkent region, Uzbekistan;
				\item[$^c$] National University of Uzbekistan,  4, Universitet str., 100174, Tashkent, Uzbekistan.
		\end{itemize}}
		\email{rozikovu@yandex.ru}
		
		\begin{abstract}
		In this short note we study a dynamical system generated by a two-parametric quadratic operator mapping 3-dimensional simplex to itself. This is an evolution operator of the frequencies of gametes in a two-locus system. We find the set of all (a continuum set) fixed points and show that each fixed point is non-hyperbolic. We completely describe the set of all limit points of the dynamical system. Namely, for any initial point (taken from the 3-dimensional simplex) we find an invariant set containing the initial point and a unique fixed point of the operator, such that the trajectory of the initial point converges to this fixed point. 		
		\end{abstract}
		\maketitle

		
		\section{Introduction}
		
In this paper following \cite[page 68]{E} we define an evolution operator of a population assuming viability selection, random mating and discrete non-overlapping generations. Consider two loci $A$ (with alleles $A_1$, $A_2$) and $B$ (with alleles $B_1$, $B_2$). Then we have four gametes: $A_1B_1$, $A_1B_2$, $A_2B_1$, and $A_2B_2$. Denote the frequencies of these gametes by $x$, $y$, $u$, and $v$ respectively.  
Thus the vector $(x, y, u, v)$ can be considered as a state of the system, and therefore, one takes it as a probability distribution on the set of gametes, i.e. as an element of  3-dimensional simplex, $S^3$.
Recall that $(m-1)$-dimensional simplex is defined as
		$$ S^{m-1}=\{x=(x_1,...,x_m)\in \mathbb R^m: x_i\geq 0,
		\sum^m_{i=1}x_i=1 \}.$$
		
Following \cite[Section 2.10]{E} we define the frequencies $(x', y', u', v')$ in the next generation as 
\begin{equation}\label{yq}	
W:	\begin{array}{llll}
x'=x+a\cdot (yu-xv)\\[2mm]
y'=y-a\cdot (yu-xv)\\[2mm]
u'=u-b\cdot (yu-xv)\\[2mm]
v'=v+b\cdot (yu-xv),
\end{array}\end{equation}
where $a,b\in [0,1]$.

It is easy to see that this quadratic operator, $W$, maps $S^3$ to itself. Indeed, we have $x'+y'+u'+v'=1$ and each coordinate is non-negative, for example, we check it for $y'$:
$$ y'=y-a\cdot (yu-xv)=y(1-au)+axv\geq y(1-au)\geq 0,$$
these inequalities follow from the conditions that $x,y,u,v,a\in [0,1]$, and therefore, we have $0\leq au\leq 1$.

The operator (\ref{yq}), for any initial point (state)  $t_0=(x_0, y_0, u_0, v_0)\in S^3$, defines its trajectory:
$\{t_n=(x_n, y_n, u_n, v_n)\}_{n=0}^\infty$ as
$$t_{n}=(x_{n}, y_{n}, u_{n}, v_{n})=W^{n}(t_0), n=0,1,2,...$$
Here $W^n$ is the $n$-fold composition of $W$ with itself:
$$W^n(\cdot)=\underbrace{W(W(W\dots (W}_{n \,{\rm times}}(\cdot)))\dots).$$ 

{\bf The main problem} in theory of dynamical system (see \cite{De}) is to study the sequence  $\{t_n\}_{n=0}^\infty$ for each initial point $t_0\in S^3$.

In general, if a dynamical system is generated by a nonlinear operator then complete solution of the main problem may be very difficult. But in  this short note we will completely solve this main problem for nonlinear operator (\ref{yq}).

\begin{rk}
	Using $1=x+y+u+v$ (on $S^3$) one can rewrite operator (\ref{yq}) as 
	\begin{equation}\label{yqs}	
		W:	\begin{array}{llll}
			x'=x^2+xy+xu+(1-a)xv+ayu\\[2mm]
			y'=xy+y^2+(1-a)yu+axv+yv\\[2mm]
			u'=xu+(1-b)yu+u^2+bxv+uv\\[2mm]
			v'=(1-b)xv+yv+byu+uv+v^2.
	\end{array}\end{equation}
	Note that the operator (\ref{yqs}) is in the form of quadratic stochastic operator (QSO), i.e., 
	$V: S^{m-1}\to S^{m-1}$ defined by 
	$$V: x_k'=\sum_{i,j=1}^mP_{ij,k}x_ix_j,$$
	where $P_{ij,k}\geq 0$, $\sum_kP_{ij,k}=1$.
	
	The operator is not studied in general, but some large class of QSO's are studied (see for example \cite{GMR}, \cite{L}, \cite{Rpd}, \cite{RS}, \cite{RZ}, \cite{RZh} and the references therein). But the operator (\ref{yq}) was not studied yet.  
\end{rk}
\section{The set of limit points}

\begin{rk} The case $a=b=0$ is very trivial, so we will not consider this case.\end{rk}
Recall that a point $t\in S^3$ is called a fixed point for $W: S^3\to S^3$ if $W(t)=t$. 

Denote the set of all fixed
points by Fix$(W)$.

It is easy to see that for any $a, b\in [0,1]$, $a+b\ne 0$ the set of all fixed points of (\ref{yq}) is 
$${\rm Fix}(W)=\{t=(x,y,u,v)\in S^3: yu-xv=0\}.$$
This is a continuum set of fixed points. 

The main problem is completely solved in the following result: 
\begin{thm}\label{tm} For any initial point $(x_0, y_0, u_0, v_0)\in S^3$  the following assertions hold 
	\begin{itemize}
		\item[1.] If $(x_0+y_0)(u_0+v_0)=0$ then $(x_0, y_0, u_0, v_0)$ is fixed point.
		\item[2.] If $(x_0+y_0)(u_0+v_0)\ne 0$ then trajectory has the following limit:
		$$		\lim_{n\to\infty}(x_{n}, y_n, u_{n}, v_n)=$$
		$$\left(A(x_0, u_0)(x_0+y_0), A(y_0, v_0)(x_0+y_0), A(x_0, u_0)(u_0+v_0), A(y_0, v_0)(u_0+v_0)\right)\in {\rm Fix}(W), 
		$$
		where 
		$$A(x,u)={bx+au\over (u_0+v_0)a+(x_0+y_0) b}.$$
	\end{itemize}	
	
\end{thm}
\begin{proof}
We note that for each $\alpha\in [0,1]$ the following set is invariant:
$$X_{\alpha}=\{t=(x,y,u,v)\in S^3: x+y=\alpha, \ \ u+v=1-\alpha\},$$
i.e., $W(X_\alpha)\subset X_\alpha$.

Note also that 
$$S^3=\bigcup_{\alpha\in [0,1]} X_\alpha.$$

The part 1 of theorem follows in the case $\alpha=0$ and $\alpha=1$. Indeed, for $\alpha=0$ we have 
	$$
	X_{0}=\{t=(0,0,u,v)\in S^3: u+v=1\},$$ 
	and in the case of $\alpha=1$ we get 
		$$
	X_{1}=\{t=(x,y,0,0)\in S^3: x+y=1\}.$$ 
	Note that in both case the restriction of operator on the corresponding set is an id-operator, i.e., all points of the set are fixed points.  
	
Now to prove part 2 we consider the case $\alpha\in (0,1)$.

Since $X_\alpha$ is an invariant, it suffices to study limit points of the operator on sets $X_\alpha$, for each $\alpha\in (0,1)$ separately. 
To do this, we reduce operator $W$ on the invariant set $X_\alpha$ (i.e., replace $y=\alpha-x$, $v=1-\alpha-u$):
\begin{equation}\label{ya}	
	W_\alpha:	\begin{array}{ll}
		x'=(1-a+a\alpha)x+a\alpha u\\[2mm]
		u'=(1-\alpha)bx+(1-b\alpha)u,
\end{array}\end{equation}
where $a,b \in [0,1]$, $\alpha\in (0,1)$ $x\in [0,\alpha]$, $u\in [0, 1-\alpha]$.

It is easy to find the set of all fixed points: 
$${\rm Fix}(W_{\alpha})=\{(x,u)\in [0,\alpha]\times [0,1-\alpha]: (1-\alpha)x-\alpha u=0\}.$$

The operator $W_\alpha$ is a linear operator given by the matrix
\begin{equation}\label{m}	
M_\alpha=\left(	\begin{array}{cc}
		1-a+a\alpha &a\alpha \\[2mm]
		(1-\alpha)b &1-b\alpha
\end{array}\right).
\end{equation}

Eigenvalues of the linear operator are 
\begin{equation}\label{ev}	\lambda_1=1, \ \ \lambda_2=1-(1-\alpha)a-\alpha b.
	\end{equation} 

For any $a,b \in [0,1]$, $a+b\ne 0$, $\alpha\in (0,1)$ we have $0< (1-\alpha)a+\alpha b< 1$, therefore,  $0<\lambda_2 < 1.$

By  (\ref{ya}) we define trajectory of an initial point $(x_0, u_0)$ as   
$$(x_{n+1}, u_{n+1})=M_\alpha (x_n, u_n)^T, \ \ n\geq 0.$$
Thus
\begin{equation}\label{sh}
	(x_{n}, u_{n})=M_\alpha^n\, (x_0, u_0)^T, \ \ n\geq 1.
\end{equation}
Therefore we need to find $M_\alpha^n$. To find it we use a little Cayley-Hamilton Theorem\footnote{https://www.freemathhelp.com/forum/threads/formula-for-matrix-raised-to-power-n.55028/}
to obtain the following formula
$$
M_\alpha^n={\lambda_2 \, \lambda^n_1- \lambda_1 \, \lambda^n_2\over \lambda_2-\lambda_1}\cdot I_2+ {\lambda^n_2- \lambda^n_1\over \lambda_2-\lambda_1}\cdot M_\alpha,
$$
where $I_2$ is $2\times 2$ unit matrix and $\lambda_1, \lambda_2$ are eigenvalues (defined in (\ref{ev})). 

By explicit formula (\ref{ev}) we get the following limit
$$
\lim_{n\to \infty}M_\alpha^n={\lambda_2 \over \lambda_2-\lambda_1}\cdot I_2- {1\over \lambda_2-\lambda_1}\cdot M_\alpha={1\over (1-\alpha)a+\alpha b}\cdot \left(	\begin{array}{cc}
	\alpha b &\alpha a \\[2mm]
	(1-\alpha)b &(1-\alpha)a
\end{array}\right).
$$
Using this limit, for any initial point $(x_0, u_0)\in [0,\alpha]\times [0,1-\alpha]$ we get
\begin{equation}\label{s}
\lim_{n\to\infty}(x_{n}, u_{n})=\lim_{n\to\infty}M_\alpha^n\, (x_0, u_0)^T={bx_0+au_0\over (1-\alpha)a+\alpha b}\cdot(\alpha, 1-\alpha)\in {\rm Fix}(W_\alpha).
\end{equation}
By (\ref{s}) we obtain
\begin{lemma} For any initial point $(x_0, y_0, u_0, v_0)\in S^3\setminus (X_0\cup X_1)$ there exists $\alpha\in (0,1)$ such that $(x_0, y_0, u_0, v_0)\in X_\alpha$ and the trajectory of this initial point (under operator $W$, defined in (\ref{yq})) has the following limit
$$		\lim_{n\to\infty}(x_{n}, y_n, u_{n}, v_n)=$$
		$$\left(A(x_0, u_0)\alpha, A(y_0, v_0)\alpha, A(x_0, u_0)(1-\alpha), A(y_0, v_0)(1-\alpha)\right)\in {\rm Fix}(W), 
	$$
	where 
	$$A(x,u)={bx+au\over (1-\alpha)a+\alpha b}.$$
\end{lemma}

In this lemma we note that $\alpha=x_0+y_0$ and $1-\alpha=u_0+v_0$, therefore, the part 2 of Theorem follows, where limit point of trajectory of each initial point is given as function of the initial point only. Theorem is proved. 
\end{proof}

\section{Biological interpretations}
The results of Theorem \ref{tm} have the following biological interpretations:

	Let $t=(x_0, y_0, u_0, v_0)\in \mathcal S^{3}$ be an initial state (the probability distribution on the set $\{A_1B_1, A_1B_2, A_2B_1, A_2B_2\}$ of gametes).
Theorem \ref{tm} says that, as a rule, the population tends to an equilibrium state with the passage of time.

 Part 1 of Theorem \ref{tm} means that if at an initial time we had only two gametes then the (initial) state remains unchanged. 
		
Part 2	means that depending on the initial state future of the population is stable: gametes survive with probability
$$	A(x_0, u_0)(x_0+y_0), A(y_0, v_0)(x_0+y_0), A(x_0, u_0)(u_0+v_0), A(y_0, v_0)(u_0+v_0)$$
respectively. 
From the existence of the limit point of any trajectory and from the explicit form of ${\rm Fix}(W)$ it follows that 
$$\lim_{n\to \infty}(y_nu_n-x_nv_n)=0.$$  
This property, biologically means (\cite[page 69]{E}), that the 
population asymptotically goes to a state of linkage equilibrium with respect to two loci. 
\section*{ Acknowledgements}
Rozikov thanks Institut des Hautes \'Etudes Scientifiques (IHES), Bures-sur-Yvette, France for support of his visit to IHES. 
The work was partially supported by a grant from the IMU-CDC.

\end{document}